\documentclass[11pt]{article}
\usepackage{amsmath}
\usepackage{amssymb}
\usepackage{graphicx}
\usepackage{color}
\usepackage{amsthm}
\usepackage{amsfonts}
\usepackage{amscd}
\usepackage{mathrsfs}
\usepackage{bm}
\usepackage{enumerate}
\usepackage{algorithmic, algorithm}

\newcommand{\D}{\displaystyle}

\newcommand{\bx}{\bm{x}}

\newcommand{\bn}{\mathbf{n}}
\newcommand{\by}{\bm{y}}
\newcommand{\bp}{\bm{p}}

\newcommand{\bs}{\bm{s}}
\newcommand{\p}{\partial}

\newcommand{\e}{\epsilon}

\newcommand{\rhk}{\bar{R}_t(\bx, \by)}

\newcommand{\bz}{\bm{z}}

\newcommand{\M}{\mathcal{M}}

\newcommand{\mathd}{\mathrm{d}}

 \newtheorem{theorem}{\textbf{Theorem}}[section]

 \newtheorem{assumption}{\textbf{Assumption}}
\newtheorem{corollary}{\textbf{Corollary}}[section]

\usepackage{soul}

\numberwithin{equation}{section}

\makeatother

\begin{document}

\title{Error estimation of weighted nonlocal Laplacian on random point cloud}

\author{
Zuoqiang Shi%
\thanks{Department of Mathematical Sciences \& Yau Mathematical Sciences Center, Tsinghua University, Beijing, China,
100084. \textit{Email: zqshi@tsinghua.edu.cn.}}%
\and
Bao Wang %
\thanks{Department of Mathematics, University of California, Los Angeles, CA 90095, USA,
 \textit{Email: wangbaonj@gmail.com}}
\and
Stanley Osher
\thanks{Department of Mathematics, University of California, Los Angeles, CA 90095, USA,
 \textit{Email: sjo@math.ucla.edu}}
}
\maketitle
\begin{abstract}
  We analyze convergence of the weighted nonlocal Laplacian (WNLL) on high dimensional randomly distributed data. 
The analysis reveals the importance of the scaling weight $\mu \sim P|/|S|$ with $|P|$ and $|S|$ be the number of entire and labeled data, respectively. The result gives theoretical foundation of WNLL for high dimensional data interpolation.
\end{abstract}

\vspace{0.1in}

\section{Introduction}

In this paper, we consider convergence of the weighted nonlocal Laplacian (WNLL) on high dimensional randomly distributed data.
WNLL is proposed in \cite{WGL} for high dimensional point cloud interpolation.
High dimensional point cloud interpolation is a fundamental problem in machine learning which can be formulated as: Let
$P = \{\bp_1, \cdots, \bp_n\}$ and $S =\{\bs_1, \cdots, \bs_m\}$ be two sets of points in $\mathbb{R}^d$
Suppose $u$ is a function defined on the point cloud $\bar{P}=P\cup S$ which is known only over $S$, denoted as $b(\bs)$ for any $\bs \in S$. The interpolation methods are used to compute $u$ over the whole point cloud $\bar{P}$ from the given values over $S$.



In nonlocal Laplacian, which is widely used in nonlocal methods for image processing \cite{BCM05,BCM06,GO07,GO08}, the interpolation function is obtained by minimizing the energy functional
\begin{equation}
\label{eq:ob-gl}
\mathcal{J}(u) = \frac{1}{2}\sum_{\bx,\by\in \bar{P}} w(\bx,\by)(u(\bx)-u(\by))^2,
\end{equation}
with the constraint
\begin{equation}
  \label{eq:constraint}
  u(\bx)=b(\bx),\quad \bx\in S.
\end{equation}
Here $w(\bx,\by)$ is a given weight function, typically chosen to be Gaussian, i.e.,
$w(\bx,\by)=\exp(-\frac{\|\bx-\by\|^2}{\sigma^2})$, $\sigma$ is a parameter, $\|\cdot\|$ is the Euclidean norm in $\mathbb{R}^d$.
In graph theory and machine learning literatures, nonlocal Laplacian is also called graph Laplacian \cite{Chung:1997,ZhuGL03}.

Graph Laplacian works very well with high labeling rate, i.e., there is a large portion of data been labeled. However, when the labeling rate is low, i.e., $|S|/|\bar{P}|\ll 1$, the solution of the graph Laplacian is found to be discontinuous at the labeled points \cite{Shi:harmonic,WGL}.
WNLL is proposed to fix this problem. In WNLL, energy functional in \eqref{eq:ob-gl} is modified by adding a weight, $\frac{|\bar{P}|}{|S|}$, to balance the labeled and unlabeled terms, which leads to
  \begin{align}
\label{opt:wgl}
\hspace{-3mm}  \min_{u} \sum_{\bx\in P}\left(\sum_{\by\in \bar{P}} w(\bx,\by)(u(\bx)-u(\by))^2\right)
+\frac{|\bar{P}|}{|S|}\sum_{\bx\in  S}\left(\sum_{\by\in \bar{P}} w(\bx,\by)(u(\bx)-u(\by))^2\right),
\end{align}
with the constraint
\begin{align*}
  u(\bx)=b(\bx),\quad \bx\in S.
\end{align*}
 When the labeling rate is high, WNLL is close to graph Laplacian.
When the labeling rate is low, the weight forces the solution to be close to the given values near the labeled points, such that the discontinuities are removed.
With a symmetric weight function, i.e. $w(\bx,\by)=w(\by,\bx)$, the corresponding Euler-Lagrange equation of \eqref{opt:wgl} is a simple linear system
 \begin{align}
2\sum_{\by\in P} w(\bx,\by)\left(u(\bx)-u(\by)\right)
+ \left(\frac{|P|}{|S|}+2\right)\sum_{\by\in S} w(\by,\bx)(u(\bx)-b(\by))&=0,\quad \bx\in P,\nonumber\\
u(\bx)&=b(\bx),\quad \bx\in S.\nonumber
\end{align}
This linear system can be solved efficiently by conjugate gradient iteration. The superiority of the WNLL compared to the graph Laplacian has been shown evidently in image inpainting \cite{Shi:harmonic,WGL}, scientific data interpolation \cite{Zhu:2018JCP},  and more recently deep learning \cite{Wang:2018}.

\subsection{Main Result}
We consider error of the WNLL in a model problem. The whole computational domain is set to be a $k$-dimensional  closed manifold $\M$ embedded in $\mathbb{R}^d$. The point cloud $P$ gives a discrete representation
of $\M$ which is assumed to be uniformly distributed on $\M$.
$\mathcal{D}\subset \M$ be a subset of $\M$ which has been labeled, and $S$ is a uniform sample of $\mathcal{D}$. In $S$, we have $u(\bx)=b(\bx)$.
An illustration of the computational domain and the point cloud is shown in Fig. \ref{fig:domain_large}.
In WNLL, in order to extend the label function $u$ to the entire dataset $P$,
we solve the linear system
\begin{align}
\label{eq:wgl}
\sum_{\by\in P} R_\delta(\bx,\by)\left(u_\delta(\bx)-u_\delta(\by)\right)+ \mu \sum_{\by\in S} K_\delta(\bx,\by)(u_\delta(\bx)-b(\by))&=0,\quad \bx\in P,\\
u_\delta(\bx)&=b(\bx),\quad \bx\in S.\nonumber
\end{align}
where $R_\delta(\bx,\by)$, $K_\delta(\bx,\by)$ are kernel functions given as
\begin{equation}
\label{eq:kernel}
R_\delta(\bx, \by) = C_\delta R\left(\frac{|\bx -\by|^2}{4\delta^2}\right),\quad K_\delta(\bx, \by) = C_\delta K\left(\frac{|\bx -\by|^2}{4\delta^2}\right),
\end{equation}
where $C_\delta = \frac{1}{(4\pi \delta^2)^{k/2}}$ is the normalization factor.
 $R, K\in C^2(\mathbb{R}^+) $ are two kernel functions satisfying the conditions listed in Assumption \ref{assumptions}.
\begin{figure}[H]
\centering
\begin{tabular}{c}
\includegraphics[width=0.6\textwidth]{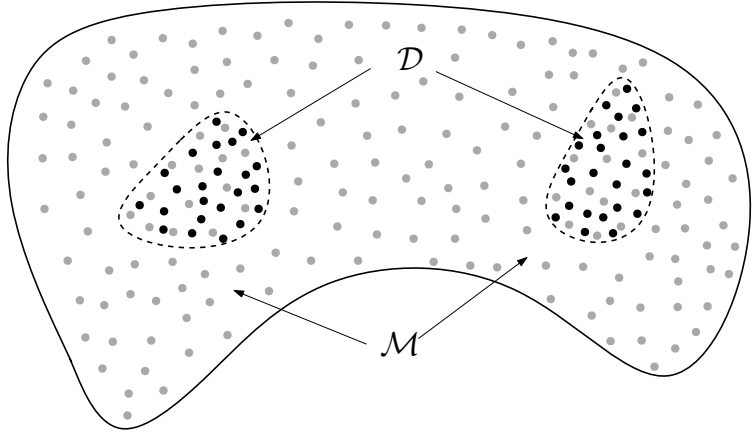}
\end{tabular}
\caption{Illustration of the computational domain. Gray points: sample of $\M$; Black points: sample of $\mathcal{D}\subset \M$.}\label{fig:domain_large}
\end{figure}
As the continuous counterpart, we consider the Laplace-Beltrami equation on a closed smooth manifold $\M$
\begin{align}
\label{eq:laplace-large}
\left\{
\begin{array}{rcll}
  \Delta_\M u(\bx)&=&0, & \bx\in \M,\\
u(\bx)&=&b(\bx),& \bx\in \mathcal{D},
\end{array}\right.
\end{align}
where $\Delta_\mathcal{M}=\text{div}(\nabla)$ is the Laplace-Beltrami operator on $\mathcal{M}$. Let $\Phi: \Omega\subset \mathbb{R}^k\rightarrow \M\subset\mathbb{R}^d$ be a local parametrization of $\M$ and $\theta\in \Omega$.
For any differentiable function $f:\M\rightarrow \mathbb{R}$,
we define the gradient on the manifold
\begin{align}
  \label{eq:diff-M}
  \nabla f(\Phi(\theta))&=\sum_{i,j=1}^m g^{ij}(\theta)\frac{\p \Phi}{\p\theta_i}(\theta)\frac{\p f(\Phi(\theta))}{\p\theta_j}(\theta).
\end{align}
And for vector field $F:\M\rightarrow T_{\bx}\M$ on $\M$, where $T_{\bx}\M$ is the tangent space of $\M$ at $\bx\in \M$, the divergence is defined as
\begin{align}
\label{eq:diver}
\text{div} (F)&= \frac{1}{\sqrt{\det G}}\sum_{k=1}^d\sum_{i,j=1}^m\frac{\p}{\p \theta_i}\left(\sqrt{\det G}\,g^{ij}F^k(\Phi(\theta))\frac{\p \Phi^k}{\p\theta_j}\right)
\end{align}
where $(g^{ij})_{i,j=1,\cdots,k}=G^{-1}$, $\det G$ is the determinant of matrix $G$ and $G(\theta)=(g_{ij})_{i,j=1,\cdots,k}$ is the first fundamental form with
\begin{eqnarray}
  \label{eq:remainn}
  g_{ij}(\theta)=\sum_{k=1}^d\frac{\p \Phi_k}{\p\theta_i}(\theta)\frac{\p \Phi_k}{\p\theta_j}(\theta),\quad i,j=1,\cdots,m.
\end{eqnarray}
and $(F^1(\bx),\cdots,F^d(\bx))^T$ is the representation of $F$ in the embedding coordinates.

To prove the convergence, we need the following assumptions.
\begin{assumption}
\label{assumptions}
\begin{itemize}
\item[]


\item \rm Assumptions on the manifold: $\M$ be a $k$-dimensional closed $C^\infty$ manifold isometrically embedded in a Euclidean space $\mathbb{R}^d$. $\mathcal{D}$ and $\p\mathcal{D}$ are smooth submanifolds of $\mathbb{R}^d$. Moreover, $b(\bx)\in C^1(\mathcal{D})$.

\item \rm Assumptions on the kernel functions:
\begin{itemize}
\item[\rm (a)] \rm Smoothness: $K(r), R(r)\in C^2(\mathbb{R}^+)$;
\item[(b)] Nonnegativity: $R(r), K(r)\ge 0$ for any $r\ge 0$.
\item[(c)] Compact support:
$R(r) = 0$ for $\forall r >1$; $K(r) = 0$ for $\forall r > r_0\ge 2$.
\item[(d)] Nondegeneracy:
 $\exists \delta_0>0$ such that $R(r)\ge\delta_0$ for $0\le r\le 1/2$ and $K(r)\ge\delta_0$ for $0\le r\le 2$.
\end{itemize}
\item Assumptions on the point cloud: $P$ and $S$ are uniformly distributed on $\M$ and $\mathcal{D}$, respectively.
\end{itemize}
\end{assumption}
In this paper, we use the notation $C$ to denote any constant which may be different in different places.
The main contribution of this paper is to analyze relation between
the solutions of the Laplace-Beltrami equation \eqref{eq:laplace-large} and the WNLL \eqref{eq:wgl}. More precisely, we prove the following theorem:
\begin{theorem}
\label{thm:main}
Let $u_\delta$ solves \eqref{eq:wgl} and $u$ solves \eqref{eq:laplace-large}.
  Under the assumptions in Assumption \ref{assumptions}, with probability at least $1-1/(2n),\; where \  n=|P|$, we have
\begin{equation}
    |u_\delta-u|\le C\delta,\nonumber
\end{equation}
as long as
\begin{equation}
\label{eq:condition-mu-intro}
    \mu\sum_{\by\in S} K_\delta(\bx, \by)\ge C \sum_{\by\in P} R_\delta(\bx,\by),\quad \bx\in P\cap\mathcal{D}_\delta.
\end{equation}
$\mathcal{D}_\delta=\{\bx\in \M: \text{dist}(\bx,\mathcal{D})\le 2\delta\}$, $C=C(\M,\mathcal{D},R,K)>0$ is a constant independent of $\delta$, $P$ and $S$.
\end{theorem}
In the above theorem, \eqref{eq:condition-mu-intro} actually gives a condition for the weight $\mu$.
Notice that $$ \frac{1}{n} \sum_{\by\in P} R_\delta(\bx,\by)\approx \frac{1}{|\M| }\int_\M R_\delta(\bx,\by)\mathd \by= O(1),\quad \bx\in P\cap\mathcal{D}_\delta.$$
$S$ samples $\mathcal{D}$, if $S$ is dense enough, we have that
\begin{equation}
\frac{1}{|S|}\sum_{\by\in S} K_\delta(\bx, \by)\approx\frac{1}{|\mathcal{D}| }\int_\mathcal{D} K_\delta(\bx,\by)\mathd \by, \quad \bx\in P\cap\mathcal{D}_\delta.\nonumber
\end{equation}
Here, we need the assumption on $K$ such that $K(r)\ge\delta_0>0, \forall 0\le r\le 2$. This implies that
\begin{equation*}
\int_\mathcal{D} K_\delta(\bx,\by)\mathd \by=O(1), \quad \bx\in P\cap\mathcal{D}_\delta.
\end{equation*}
Hence, from \eqref{eq:condition-mu-intro}, we have
\begin{equation}
    \mu\sim \frac{|P|}{|S|}.\nonumber
\end{equation}
This explains the scaling of $\mu$ in WNLL.
\begin{figure}[H]
\centering
\begin{tabular}{c}
\includegraphics[width=0.6\textwidth]{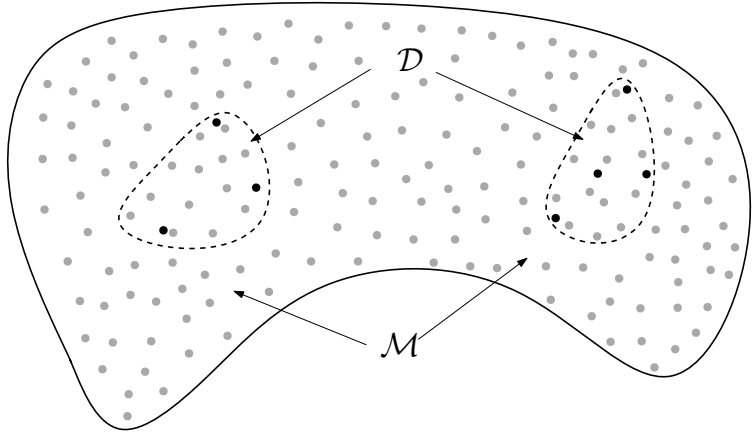}
\end{tabular}
\caption{Illustration of the computational domain with extremely low labeling rate.}\label{fig:domain_sparse}
\end{figure}
On the other hand, if sample $S$ is extremely sparse such that $\bigcup_{\bx\in S}B(\bx;4\delta)$ does not cover $\mathcal{D}_\delta$
 as shown in Fig. \ref{fig:domain_sparse}, $\sum_{\by\in S} K_\delta(\bx, \by)$ may be zero for some $\bx\in P\cap\mathcal{D}_\delta$.
In this case,  condition \eqref{eq:condition-mu-intro} does not hold. Then we can not guarantee the convergence even in WNLL.
With extremely low labeling rate, actually, the whole framework of harmonic extension fails \cite{NS09,ZB11}. We should use other approach to get a smooth interpolation.

Theorem \ref{thm:main} is a direct consequence of the maximum principle (Theorem \ref{thm:maximum-principle}) and the error estimation (Theorem \ref{thm:error}).
\begin{theorem}
\label{thm:maximum-principle}
 Under the assumptions in Assumption \ref{assumptions}, with probability at least $1-1/(2n)$, $n=|P|$, $L_{\delta,n}$ has the comparison principle, i.e.
\begin{align}
  |L_{\delta,n} u (\bx)|\le L_{\delta,n} v (\bx)\quad \rightarrow\quad |u|\le v,\nonumber
\end{align}
where
 \begin{align}
\label{eq:Ld}
L_{\delta,n} u(\bx)=\sum_{\by\in P} R_\delta(\bx,\by)\left(u(\bx)-u(\by)\right)+ \mu \sum_{\by\in S} K_\delta(\bx,\by)u(\bx),\quad \bx\in P.
\end{align}
\end{theorem}

\begin{theorem}
\label{thm:error}
Let $u_\delta$ and $u$ solve \eqref{eq:wgl} and \eqref{eq:laplace-large} respectively. $v$ is the solution of \eqref{eq:laplace-large-ref},
\begin{align}
  \label{eq:laplace-large-ref}
\left\{
\begin{array}{rcll}
  -\Delta_\M v(\bx)&=&1, & \bx\in \M\backslash \mathcal{D},\\
v(\bx)&=&1,& \bx\in \mathcal{D}.
\end{array}\right.
\end{align}
 Under the assumptions in Assumption \ref{assumptions}, with probability at least $1-1/(2n)$, $n=|P|$,
\begin{align}
  |L_{\delta,n} (u_\delta-u)|\le C \delta L_{\delta,n} v,\nonumber
\end{align}
as long as
\begin{equation}
    \mu\sum_{\by\in S} K_\delta(\bx, \by)\ge C \sum_{\by\in P} R_\delta(\bx,\by),\quad \bx\in P\cap\mathcal{D}_\delta.\nonumber
\end{equation}
$\mathcal{D}_\delta=\{\bx\in \M: \text{dist}(\bx,\mathcal{D})\le 2\delta\}$, $C=C(\M,\mathcal{D},R,K)>0$ is a constant independent on $\delta$, $P$ and $S$.
\end{theorem}

The above two theorems will be proved in Section \ref{sec:max} and Section \ref{sec:convergence}, respectively.
In Section \ref{sec:entropy}, we prove a technical theorem used in the analysis. Some discussions are made in Section \ref{sec:concl}.

\section{Maximum Principle (Theorem \ref{thm:maximum-principle})}
\label{sec:max}
First, we introduce some notations.
For any two points $\bx,\by\in P$, we say that they are neighbors if and only if $R_\delta(\bx,\by)>0$, denoted as $\bx\sim \by$.
For $\bx\in S,\by\in P\cup S$, they are neighbors if and only if $K_\delta(\bx,\by)>0$, denoted also by $\bx\sim \by$ or $\by\sim \bx$. $\bx$ and $\by$ are connected if there exist
$\bz_1,\cdots,\bz_m\in P\cup S$ such that $$\bx\sim \bz_1\sim\cdots\sim\bz_m\sim\by.$$
We say point cloud $P$ is $S$-connected if for any point $\bx\in P$, there exists $\by\in S$, such that $\bx$ and $\by$ are connected.

If $P$ is $S$-connected, it is easy to check that $L_{\delta,n}$ has the maximum principle, i.e.
\begin{align}
  \label{eq:max}
  L_{\delta,n} u (\bx)\ge 0,\; \bx \in P \quad \rightarrow \quad u(\bx)\ge 0,\; \bx \in P,\\
  L_{\delta,n} u (\bx)\le 0,\; \bx \in P \quad \rightarrow \quad u(\bx)\le 0,\; \bx \in P.
\end{align}
and consequently
\begin{align}
  \label{eq:compare}
  |L_{\delta,n} u (\bx)|\le L_{\delta,n} v (\bx)\quad \rightarrow\quad |u|\le v.
\end{align}
In the rest of this section, we will prove that with high probability, $P$ is $S$-connected. To prove this, we need a theorem from the empirical process theory \cite{entropy}.
\begin{theorem}
\label{thm:entropy-bound}
  With probability at least $1-1/(2n)$, $n=|P|$,
\begin{align}
  \sup_{f\in \mathcal{R}_\delta}|I(f)-I_n(f)|\le \frac{C}{\delta^k\sqrt{n}}\left(\ln n-2\ln \delta +1\right)^{1/2},
\end{align}
where $k$ is the dimension of $\M$,
$$I(f)=\frac{1}{|\M|}\int_\M f(\bx)\mathd \bx,\quad I_n(f)=\frac{1}{n}\sum_{\bx\in P}f(\bx),$$
$|\M|$ is the volume of $\M$ and $\mathcal{R}_\delta$ is a function class defined as
$$\mathcal{R}_\delta=\{R_\delta(\bx,\cdot): \bx\in \M.\}$$
\end{theorem}
This theorem will be proved in Section \ref{sec:entropy}.

Suppose $P$ is not $S$-connected. Let $$\bar{S}=\{\bx\in P\cup S: \bx \;\text{is connected to $S$}\},\quad \bar{S}^c=(P\cup S)\backslash \bar{S}.$$
Then $\bar{S}^c\ne \emptyset$.
Denote
\begin{align*}
  \bar{S}_\delta=\left(\bigcup_{\bx\in \bar{S}}B(\bx;\delta/2)\right)\cap \M,\quad \bar{S}_\delta^c=\left(\bigcup_{\bx\in \bar{S}^c}B(\bx;\delta/2)\right)\cap \M
\end{align*}
where $B(\bx;\delta)=\{\by\in \mathbb{R}^d: |\bx-\by|\le \delta\}$.

Using the definition of $\bar{S}$ and $\bar{S}^c$, we know that
$\bar{S}_\delta\cap \bar{S}_\delta^c=\emptyset$, hence
$$\p \bar{S}_\delta\cap \bar{S}_\delta^c=\emptyset,$$
where $\p \bar{S}_\delta$ is the boundary of $\bar{S}_\delta$ in $\mathbb{R}^d$.
Furthermore, since $\M$ is connected, we have
$$\p \bar{S}_\delta \cap \M \ne \emptyset.$$
Choose any $\bx_0\in \p \bar{S}_\delta \cap \M$, we also have that $\bx_0\notin \bar{S}_\delta^c$, which implies that
$$R_{\delta/4}(\bx_0,\by)=0,\quad \forall \by\in P.$$
It follows that
$$I_n(R_{\delta/4}(\bx_0,\cdot))=0.$$
On the other hand, $I(R_{\delta/4}(\bx_0,\cdot))=O(1)$. Using Theorem \ref{thm:entropy-bound}, we know that the probability is less than $1/(2n)$, which proves that
$P$ is $S$-connected with probability at least $1-1/(2n)$.
So far, we have proved Theorem \ref{thm:maximum-principle}.

\section{Error Estimate (Theorem \ref{thm:error})}
\label{sec:convergence}

Let $e_\delta(\bx)=u_\delta(\bx)-u(\bx)$. $u_\delta$ and $u$ solve \eqref{eq:wgl} and \eqref{eq:laplace-large} respectively.

Direct calculation shows that
\begin{align}
\label{eq:error-large}
  L_{\delta,n} e_\delta(\bx)&=\sum_{\by\in P} R_\delta(\bx,\by)(u(\bx)-u(\by))+\mu\sum_{\by\in S} K_\delta(\bx, \by)(u(\bx)-b(\by))\mathd \by,
\quad \bx\in P,\\
\label{eq:error-large-2}
e_\delta(\bx)&=0,\quad \bx\in S.
\end{align}
Next, we will find an upper bound of the right hand side in \eqref{eq:error-large}.

An upper bound of the second term of \eqref{eq:error-large} is relatively easy to find by using the smoothness of $u$ and $b$:
\begin{align}
  \label{eq:error-2}
  \left|\sum_{\by\in S} K_\delta(\bx, \by)(u(\bx)-b(\by))\right|\le C\delta \sum_{\by\in S} K_\delta(\bx, \by)
\end{align}

To find an upper bound of the first term, we need the following theorem which can be found in \cite{Shi-iso}.
\begin{theorem}
 Let $u(\bx)\in C^3(\M)$ and
\begin{align}
\label{eq:error_boundary}
I_{bd} =\sum_{j=1}^d \int_{\p\M}n^j(\by)(\bx-\by)\cdot\nabla(\nabla^ju(\by))\rhk\mathd \tau_{\by},
\end{align}
and
\begin{align}
I_{in}=\frac{1}{\delta^2}\int_{\M} R_\delta(\bx,\by)(u(\bx)-u(\by))\mathd \by+\int_\M \bar{R}_\delta(\bx,\by) \Delta_\M u(\by)\mathd \by-\int_{\p\M} \bar{R}_\delta(\bx,\by)\frac{\p u}{\p \bn}(\by)
\mathd \tau_{\by}-I_{bd}\nonumber.
\end{align}
where $\bn(\by)=(n^1(\by),\cdots,n^d(\by))$ is the out normal vector of $\p\M$ at $\by$, $\nabla^j$ is the $j$th component of gradient $\nabla$,
$\bar{R}_\delta(\bx,\by)=C_\delta\bar{R}\left(\frac{|\bx-\by|^2}{4\delta^2}\right)$ and $\bar{R}(r)=\int_{r}^{\infty}R(s)\mathd s$.

Then there exist constants $C, T_0$ depending only on $\M$ and $p(\bx)$, so that,
\begin{eqnarray}
\label{eqn:integral_error_int}
\left|I_{in}\right|\le C\delta\|u\|_{C^3(\mathcal{M})},\quad
\end{eqnarray}
as long as $\delta\le T_0$.
\label{thm:integral_error}
\end{theorem}

According to the above theorem, we have
\begin{align}
  \frac{1}{\delta^2}\left|\int_{\M} R_\delta(\bx,\by)(u(\bx)-u(\by))\mathd \by\right|\le C\delta,\quad \bx\in \M\backslash \mathcal{D}_\delta,
\end{align}
where $\mathcal{D}_\delta=\{\bx\in \M: \text{dist}(\bx,\mathcal{D})\le 2\delta\}$. Notice that for $\bx\in \M\backslash \mathcal{D}$, $R_\delta(\bx,\by)$ has no intersection with
$\mathcal{D}$, so all boundary terms vanish.

To get an upper bound of the first term in \eqref{eq:error-large}, we need to estimate the difference between $\int_{\M} R_\delta(\bx,\by)(u(\bx)-u(\by))\mathd \by$
and $\sum_{\by\in P} R_\delta(\bx,\by)(u(\bx)-u(\by))$. This is given by the following theorem.
\begin{theorem}
\label{thm:entropy-bound-2}
  With probability at least $1-1/(2n)$, $n=|P|$,
\begin{align}
  \sup_{f\in \bar{\mathcal{R}}_\delta}|I(f)-I_n(f)|\le \frac{C}{\delta^k\sqrt{n}}\left(\ln n-2\ln \delta +1\right)^{1/2},
\end{align}
where $k$ is the dimension of $\M$,
$$I(f)=\frac{1}{|\M|}\int_\M f(\bx)\mathd \bx,\quad I_n(f)=\frac{1}{n}\sum_{\bx\in P}f(\bx),$$
$|\M|$ is the volume of $\M$ and $\mathcal{R}_\delta$ is a function class defined as
$$\bar{\mathcal{R}}_\delta=\{R_\delta(\bx,\cdot), R_\delta(\bx,\cdot)u(\cdot), R_\delta(\bx,\cdot)v(\cdot): \bx\in \M,\; u \;and\; v\; \text{solves \eqref{eq:laplace-large} and
\eqref{eq:laplace-large-ref} respectively.}\}$$
\end{theorem}
This theorem will be proved in Section \ref{sec:entropy} using the empirical process theory \cite{entropy}.

Using Theorem \ref{thm:entropy-bound-2}, we have
\begin{align}
\label{eq:error-1-1}
  \left|\frac{1}{n}\sum_{\by\in P} R_\delta(\bx,\by)(u(\bx)-u(\by))\right|\le \frac{C}{\delta^k\sqrt{n}}(\ln n-2\ln \delta +1)^{1/2}+C\delta^3,\quad \bx\in P\cap (\M\backslash \mathcal{D}_\delta).
\end{align}
For $\bx\in P\cap \mathcal{D}_\delta$, the bound is straightforward, just using the smoothness of $u$,
\begin{align}
\label{eq:error-1-2}
  \left|\frac{1}{n}\sum_{\by\in P} R_\delta(\bx,\by)(u(\bx)-u(\by))\right|\le \frac{C\delta}{n} \sum_{\by\in P} R_\delta(\bx,\by),\quad \bx\in P\cap \mathcal{D}_\delta.
\end{align}
Substituting \eqref{eq:error-2}, \eqref{eq:error-1-1} and \eqref{eq:error-1-2} in \eqref{eq:error-large}, we have
\begin{align}
\label{eq:error-D}
  | L_{\delta.n} e_\delta(\bx)|\le C\delta \left(\frac{1}{n}\sum_{\by\in P} R_\delta(\bx,\by)\right)+C\delta\left(\frac{\mu}{n} \sum_{\by\in S} K_\delta(\bx, \by)\right)
, \quad \bx\in P\cap \mathcal{D}_\delta.
\end{align}
and
\begin{align}
  | L_{\delta.n} e_\delta(\bx)|\le \frac{C}{\delta^k\sqrt{n}}(\ln n-2\ln \delta +1)^{1/2}+C\delta^3+C\delta\left(\frac{\mu}{n} \sum_{\by\in S} K_\delta(\bx, \by)\right),
 \quad \bx\in P\cap (\M\backslash \mathcal{D}_\delta).\nonumber
\end{align}
Suppose the number of sample points $n$ is large enough such that
\begin{equation}
\label{eq:p-large}
\frac{C}{\delta^k\sqrt{n}}(\ln n-2\ln \delta +1)^{1/2}\le \delta^3,
\end{equation}
then we have
\begin{align}
\label{eq:error-P}
  | L_{\delta.n} e_\delta(\bx)|\le C\delta^3+C\delta \left(\frac{\mu}{n}\sum_{\by\in S} K_\delta(\bx, \by)\right), \quad \bx\in P\cap (\M\backslash \mathcal{D}_\delta).
\end{align}
\eqref{eq:error-D} and \eqref{eq:error-P} give an upper bound for $| L_{\delta.n} e_\delta(\bx)|$.

Next, we want to get a lower bound of $ L_{\delta.n} v(\bx)$ with $v$ given in \eqref{eq:laplace-large-ref}. 

By Theorem \ref{thm:integral_error} and \eqref{eq:laplace-large-ref}, we have
\begin{align}
  \frac{1}{\delta^2}\int_{\M} R_\delta(\bx,\by)(v(\bx)-v(\by))\mathd \by\ge \int_{\M} \bar{R}_\delta(\bx,\by)\mathd \by -C\delta,
\quad \bx\in \M\backslash \mathcal{D}_\delta.
\end{align}
Also using Theorem \ref{thm:entropy-bound-2}
\begin{align}
  &\frac{1}{n}\sum_{\by\in P} R_\delta(\bx,\by)(v(\bx)-v(\by))\nonumber\\
  \ge& \; \bar{w}_\delta\delta^2 -C\delta^3-\frac{C}{\delta^k\sqrt{n}}(\ln n-2\ln \delta +1)^{1/2}
  \ge  \bar{w}_\delta\delta^2/2
  ,\quad \bx\in P\cap (\M\backslash \mathcal{D}_\delta).
  \label{eq:error-v1-1}
\end{align}
with $\D\bar{w}_\delta=\min_{\bx\in \M\backslash \mathcal{D}_\delta}\frac{1}{|\M|}\int_{\M} \bar{R}_\delta(\bx,\by)\mathd \by$. Here, we also use the assumption that $n$ is
large enough, \eqref{eq:p-large}.

In $P\cap \mathcal{D}_\delta$, we have
\begin{align}
\label{eq:error-v1-2}
 \frac{1}{n}\sum_{\by\in P} R_\delta(\bx,\by)(v(\bx)-v(\by))\ge &-\frac{C\delta}{n} \sum_{\by\in P} R_\delta(\bx,\by),\quad \bx\in P\cap \mathcal{D}_\delta,
\end{align}
this is due to the smoothness of $v$.

Also notice that
\begin{align}
  \label{eq:error-v2}
  \frac{1}{n}\sum_{\by\in S} K_\delta(\bx, \by)v(\bx)= \frac{v(\bx)}{n} \sum_{\by\in S} K_\delta(\bx, \by) \ge \frac{1}{n}\sum_{\by\in S} K_\delta(\bx, \by).
\end{align}
Combining \eqref{eq:error-v1-1}, \eqref{eq:error-v1-2} and \eqref{eq:error-v2}, we obtain
\begin{align}
\label{eq:error-v-D}
   L_{\delta.n} v(\bx)\ge \frac{\mu}{n}\sum_{\by\in S} K_\delta(\bx, \by)-\frac{C\delta}{n} \sum_{\by\in P} R_\delta(\bx,\by), \quad \bx\in P\cap \mathcal{D}_\delta.
\end{align}
and
\begin{align}
\label{eq:error-v-P}
   L_{\delta,n} v(\bx)\ge&\; \frac{\bar{w}_\delta}{2}\delta^2
+\frac{\mu}{n}\sum_{\by\in S} K_\delta(\bx, \by), \quad \bx\in P\cap (\M\backslash \mathcal{D}_\delta).
\end{align}
Comparing \eqref{eq:error-P} and \eqref{eq:error-v-P}, we have
\begin{equation}
    |L_{\delta,n} e_\delta(\bx)|\le C \delta L_{\delta.n} v(\bx),\quad \bx\in P\cap (\M\backslash \mathcal{D}_\delta).
\end{equation}
Meanwhile, \eqref{eq:error-D} and \eqref{eq:error-v-D} show that
\begin{equation}
    |L_{\delta,n} e_\delta(\bx)|\le C \delta L_{\delta.n} v(\bx),\quad \bx\in P\cap\mathcal{D}_\delta,
\end{equation}
as long as
\begin{equation}
\label{eq:condition-mu}
    \frac{\mu}{n}\sum_{\by\in S} K_\delta(\bx, \by)\ge \frac{C}{n} \sum_{\by\in P} R_\delta(\bx,\by),\quad \bx\in P\cap\mathcal{D}_\delta.
\end{equation}
The proof of Theorem \ref{thm:error} is completed.


\section{Entropy bound}
\label{sec:entropy}

In this section, we will prove Theorem \ref{thm:entropy-bound} and \ref{thm:entropy-bound-2}. The method we use is to estimate the covering number of the function classes.
First we introduce the definition of the covering number.

Let $(Y, d)$ be a metric space and set $F\subset Y$ . For every $\e>0$, denote by $N (\e, F, d)$ the
minimal number of open balls (with respect to the metric $d$) that are needed to cover $F$. That is,
the minimal cardinality of the set $\{y_1 , \cdots, y_m\}\subset Y$ with the property that every $f \in  F$
has some $y_i$ such that $d(f, y_i ) < \e$. The set $\{y_1 , \cdots, y_m \}$ is called an $\e$-cover of $F$ . The
logarithm of the covering numbers is called the entropy of the set.
For every sample $\{x_1 , \cdots, x_n \}$,
let $\mu_n$ be the empirical measure supported on that sample. For $1 \le p <\infty$ and a
function $f$ , put $\|f\|_{L_p (\mu_n )} =\left(\frac{1}{n}\sum_{i=1}^n  |f (x_i )|^p\right)^{1/p}$
and set $\|f\|_\infty = \max_{1\le i\le n} |f (x_i )|$. Let
$N(\e , F, L_p (\mu_n )$ be the covering numbers of $F$ at scale $\e$ with respect to the $L_p (\mu_n )$ norm.

We will use the following theorem which is well known in empirical process theory.
\begin{theorem}(Theorem 2.3 in \cite{entropy})
\label{thm:entropy}
Let $F$ be a class of functions from $\M$ to $[-1,1]$ and set $\mu$ to be a probability measure on $\M$.
Let $(\bx_i)_{i=1}^\infty$ be independent
random variables distributed according to $\mu$. For any $\e>0$ and every $n\ge 8/\e^2$,
  \begin{eqnarray}
    \mathbb{P}\left(\sup_{f\in F}|\frac{1}{n}\sum_{i=1}^nf(\bx_i)-\int_\M f(\bx)\mu(\bx)\mathd \bx|>\e\right)
\le 8\mathbb{E}_\mu[ N(\e/8,F,L_1(\mu_n))]\exp(-n\e^2/128)
  \end{eqnarray}
\end{theorem}
Note that
  \begin{align*}
    L_1(\mu_n)\le L_\infty(\mu_n)\le L_\infty
  \end{align*}
where $\|f\|_{L_\infty}=\max_{\bx\in\M}|f(\bx)|$. Then we get following corollary.
\begin{corollary}
\label{cor:entropy-0}
Let $F$ be a class of functions from $\M$ to $[-1,1]$ and set $\mu$ to be a probability measure on $\M$.
Let $(\bx_i)_{i=1}^\infty$ be independent
random variables distributed according to $\mu$. For any $\e>0$ and every $n\ge 8/\e^2$,
  \begin{eqnarray}
    \mathbb{P}\left(\sup_{f\in F}|\frac{1}{n}\sum_{i=1}^n f(\bx_i)-\int_\M f(\bx)\mu(\bx)\mathd \bx|>\e\right)
\le 8 N(\e/8,F,L_\infty)\exp(-n\e^2/128)
  \end{eqnarray}
where $N(\e , F, L_\infty)$ is the covering numbers of $F$ at scale $\e$ with respect to the $L_\infty$ norm
\end{corollary}

\begin{corollary}
\label{cor:entropy}
Let $F$ be a class of functions from $\M$ to $[-1,1]$. Let $(\bx_i)_{i=1}^\infty$ be independent
random variables distributed according to $p$, where $p$ is the probability distribution. Then
with probability at least $1-\delta$, we have
  \begin{eqnarray}
    \sup_{f\in F}|p(f)-p_n(f)|\le  \sqrt{\frac{128}{n}\left(\ln N(\sqrt{\frac{2}{n}},F,L_\infty)
+\ln \frac{8}{\delta}\right)},\nonumber
  \end{eqnarray}
where
\begin{align}
  \label{eq:note-mc}
  p(f)=\int_\M f(\bx)p(\bx)\mathd \bx,\quad p_n(f)=\frac{1}{n}\sum_{i=1}^nf(\bx_i).
\end{align}
\end{corollary}
\begin{proof}
  Using Corollary \ref{cor:entropy-0}, with probability at least $1-\delta$,
\begin{eqnarray}
    \sup_{f\in F}|p(f)-p_n(f)|\le  \e_\delta,\nonumber
  \end{eqnarray}
where $\e_\delta$ is determined by
\begin{eqnarray}
  \e_\delta=\sqrt{\frac{128}{n}\left(\ln N(\e_\delta/8,F,L_\infty)+\ln \frac{8}{\delta}\right)}.\nonumber
\end{eqnarray}
Obviously,
\begin{align*}
  \e_\delta\ge \sqrt{\frac{128}{n}}=8\sqrt{\frac{2}{n}}
\end{align*}
which gives that
\begin{align*}
  N(\e_\delta/8,F,L_\infty)\le N(\sqrt{\frac{2}{n}},F,L_\infty)
\end{align*}
Then, we have
\begin{eqnarray}
  \e_\delta \le \sqrt{\frac{128}{n}\left(\ln N(\sqrt{\frac{2}{n}},F,L_\infty)
+\ln \frac{8}{\delta}\right)}\nonumber
\end{eqnarray}
which proves the corollary.
\end{proof}

The above corollaries provide a tool to estimate the integral error on random samples. To apply the above corollaries in our problem, the key point is to obtain the estimates of the covering number of function class $\mathcal{R}_\delta$.

Since the kernel $R\in C^1(\M)$ and $\M\in C^\infty$, we have
for any $\bx,\by\in \M$
\begin{eqnarray}
  |R\left(\frac{\|\bx-\by\|^2}{4\delta^2}\right)-R\left(\frac{\|\bz-\by\|^2}{4\delta^2}\right)|\le \frac{C}{\delta}\|\bx-\bz\|.\nonumber
\end{eqnarray}
This gives an easy bound of $N(\e, \mathcal{R}_\delta,L_\infty)$,
\begin{eqnarray}
\label{eq:bound-R}
  N(\e, \mathcal{R}_\delta,L_\infty)\le \left(\frac{C}{\e\delta}\right)^k
\end{eqnarray}
Using Corollary \ref{cor:entropy}, with probability at least $1-1/(2n)$,
\begin{align}
\label{eq:bound-w-prob}
  \sup_{f\in \mathcal{R}_\delta}|p(f)-p_n(f)|\le \frac{C}{\delta^k \sqrt{n}}\left(\ln n-2\ln \delta +1\right)^{1/2}
\end{align}
This proves Theorem \ref{thm:entropy-bound}. Theorem \ref{thm:entropy-bound-2} can be proved similarly using the fact that $u$ (solution of \eqref{eq:laplace-large})
 and $v$ (solution of \eqref{eq:laplace-large-ref}) are both smooth.

\section{Discussion and Future Works}
\label{sec:concl}

In this paper, we analyzed convergence of the weighted nonlocal Laplacian (WNLL) on random point cloud. The analysis reveals that the weight is very important in the convergence and it should have the same order as $|P|/|S|$, i.e. $\mu\sim |P|/|S|$. The result in this paper provides the WNLL a solid theoretical foundation.

Furthermore, our analysis also shows that the convergence may fail with extremely low labeling rate. As discussed in Section \ref{sec:convergence}, in this case, we should consider other approaches. One interesting option is to minimize $L_\infty$ norm of the gradient instead of the $L_2$ norm, i.e. to solve the following optimization problem
\begin{align*}
  \min_u \left(\max_{\bx\in P\cup S} \left(\sum_{\by\in P\cup S}w(\bx,\by)(u(\bx)-u(\by))^2\right)^{1/2}\right),
\end{align*}
with the constraint
\begin{align*}
  u(\bx)=b(\bx),\quad \bx\in S.
\end{align*}
This approach is closely related to the infinity Laplacian \cite{GEL11,EDLL14}. The above optimization problem can be solved by the split Bregman iteration. An interesting observation
is that the WNLL can accelerate convergence of the split Bregman iteration and improve efficiency. This will be further explore in our future work.

\vspace{0.2in}
\noindent
\textbf{Acknowledgment.}
This material is based, in part, upon work supported by the U.S. Department of Energy, Office of Science and by National Science Foundation, and National Science Foundation of China, under Grant Numbers DOE-SC0013838 and DMS-1554564, (STROBE), NSFC 11671005.

\newpage

\end{document}